\documentclass[12pt]{conm-p-l}
\usepackage{amsmath}
\usepackage{amsfonts}
\usepackage{amsthm}
\usepackage{calrsfs}
\usepackage{enumerate}
\usepackage{geometry}
\usepackage{amssymb,latexsym}

\usepackage{array}
\usepackage[hidelinks]{hyperref}

\usepackage{microtype}
\usepackage[english]{babel}

\newcommand{\FF}{\mathbb F}

\newcommand{\cO}{\mathcal O}

\newcommand{\cS}{\mathcal S}
\newcommand{\cQ}{\mathcal Q}
\newcommand{\cL}{\mathcal L}

\newcommand{\cP}{\mathcal P}

\newcommand{\cC}{\mathcal C}

\newcommand{\cG}{\mathcal G}
\newcommand{\Rad}{\mathrm{Rad}\,}

\newcommand{\PG}{\mathrm{PG}}
\newcommand{\vep}{\varepsilon }
\newcommand{\PGL}{\mathrm{PGL}}

\newcommand{\Aut}{\mathrm{Aut}\,}
\newcommand{\SbS}{\mathbb S}

\newtheorem{theorem}{Theorem}[section]

\newtheorem{corollary}[theorem]{Corollary}
\newtheorem{prop}[theorem]{Proposition}

\theoremstyle{definition}
\newtheorem{remark}{Remark}

\newtheorem{definition}[theorem]{Definition}
\newcommand{\cH}{\mathcal H}

\newcommand{\cX}{\mathcal X}

\begin{document}
\title[Minimal codes from embeddings]{On minimal codes arising from projective embeddings of point-line geometries}

\author{Ilaria Cardinali}
\address{Dep. Information Engineering and Mathematics, University of Siena
  Via Roma~56, I-53100 Siena, Italy}
\email{ilaria.cardinali@unisi.it}

\author{Luca Giuzzi}
\address{D.I.C.A.T.A.M.,
Universit\`a di Brescia,
Via Branze 43, I-25123 Brescia, Italy}
\email{luca.giuzzi@unibs.it}

\begin{abstract}
Let $\cC(\Omega)$  be the linear code arising from a projective system $\Omega$ of $\PG(V).$
Consider the point-line geometry $\Gamma=(\cP,\cL)$ and a projective embedding $\vep\colon \Gamma\rightarrow \PG(V)$ of $\Gamma.$ We show that the projective code obtained  by taking as projective system  $\Omega:=\vep(\cP)$ is minimal if
    the graph induced on the set $\Gamma\setminus\vep^{-1}(H)$ by the collinearity graph of $\Gamma$ is connected for any hyperplane $H$ of $\PG(V)$.
    As an application, we prove that Grassmann codes, Segre codes, line polar Grassmann codes of orthogonal, symplectic, hermitian type, codes arising from dual polar spaces of orthogonal and symplectic type  and codes arising from the point-hyperplane geometry of a projective space
    are minimal codes.	
\end{abstract}

\maketitle

\section{Introduction}
The aim of this note is to link the theory of point-line geometries embedded into projective spaces with the theory of projective codes, focusing on minimal codes.

Ever since Shannon's introduction of coding theory~\cite{Shannon},
a key problem in coding theory has been to determine linear codes with
good parameters, i.e. having good minimum distance and dimension with
respect to their length. One of the most fruitful ideas in
this field has been to make recourse to constructions related
to geometry.
In F. Mac Williams' seminal thesis of 1962~\cite{MacWilliams1962}, it is
discussed how it is possible
to obtain a projective code starting from a set of points spanning
a finite
projective space and it is shown how the minimum distance is related to the
hyperplane sections of this set; this approach
is fully formalized in R.C. Burton thesis of 1964~\cite{Burton}.
In this note we shall follow this approach but in
the modern presentation of~\cite{TVZ}; as a standard reference on
coding theory we refer the reader to~\cite{MacWilliams1977}.

Let $V:=\FF_q^k$ be a vector space of
dimension $k$ over the finite field $\FF_q$
with $q$ elements.
A \emph{projective system} $\Omega\subseteq\PG(V)=\PG(k-1,q)$ is a set of $N$ distinct points spanning the  $(k-1)$-dimensional projective space
$\PG(V)$. Given such a projective system, it is possible to construct a $[N,k,d]$-linear code  $\cC(\Omega)$,  uniquely defined up to monomial code
equivalence, by taking the coordinate representatives (with respect to some fixed reference system) of the points of $\Omega$
as columns of a generator matrix $G$ for $\cC(\Omega).$
Such a code is called a \emph{projective code}.

The parameters $[N,k,d]$ of $\cC(\Omega)$ depend only on  the pointset
$\Omega$: the length $N$ is the size of $\Omega$ and the dimension $k$
is the (vector) dimension of the space spanned
by $\Omega$.
As any word $c$ of $\cC(\Omega)$ is of the form
$c=mG$ for some row vector $m\in\FF_q^k$,
it is straightforward to see that the number of
zeroes in $c=(c_i)_1^N$ is the same as the number of points of $\Omega$ lying on the hyperplane (of $\PG(k-1,q)$) of equation
$m\cdot x=0,$ where $m\cdot x=\sum_{i=1}^k m_ix_i$ and
$m=(m_i)_1^k$, $x=(x_i)_1^k$. So, the spectrum of the cardinalities of the complements of the intersections of $\Omega$
with the hyperplanes of $\PG(k-1,q)=\PG(\langle \Omega\rangle)$
provides the list of the weights of
$\cC(\Omega)$. In particular, the minimum Hamming distance $d$ of $\cC$ is
also the minimum Hamming weight and
\begin{equation}\label{min distance 1}
	d =N-\max_{H}|\Omega\cap H|,
\end{equation}
as $H$ ranges among all hyperplanes of  $\PG(\langle\Omega\rangle).$

This approach of constructing codes from projective systems has been  applied to describe and characterize
several families of projective codes, such as, for example,
Grassmann codes \cite{Ghorpade2001,Nogin1996},
Flag variety codes~\cite{Rodier2003, CL-long1, CL-long2},
Schubert codes
\cite{Ghorpade2013,Ghorpade2005},  polar Grassmann codes \cite{
  Cardinali2016b, Cardinali2018, Cardinali2018c, Cardinali2016a}.

J.\ Massey in~\cite{Massey} introduced the notion of \emph{minimal codewords}
(or minimal vectors) of a linear code $\cC$, primarily
in order to devise  efficient secret sharing schemes.
Minimal codewords have also numerous applications besides
cryptography; e.g.\ they
are relevant for bounding the complexity of
some decoding algorithms (see~\cite{Ashikhmin1998})
or for studying the matroid of the code~\cite{Gorla2023}.

A code is \emph{minimal} if any codeword $c\in \cC $ has the property that  for any other codeword  $c'\in\cC$ such that $\mathrm{supp}(c')\subseteq\mathrm{supp}(c)$
there exists $\lambda\in\FF_q$ such that  $c'=\lambda c.$

Ashikhmin and Barg in~\cite[Lemma 2.1, part (3)]{Ashikhmin1998} determined
a well known and widely used
sufficient condition for a code $\cC$ to be minimal
which depends only on the minimum and maximum weight
(respectively $w_{\min}$ and $w_{\max}$) of the non-null
codewords of $\cC$:
\begin{equation}
	\label{ab-bound}
	\frac{w_{\max}}{w_{\min}}<\frac{q}{q-1}.
\end{equation}

Since the use of bound~\eqref{ab-bound} requires the knowledge of the minimum and the maximum weight of the code, it is not always
easily applicable. Determining the weights of a code
is an NP-hard problem~\cite{Berlekamp1978}, so, in general,
also determining the minimal codewords for
non-minimal codes is considered to be a very hard problem.

In this paper, relying on the geometrical characterization of the minimality due to~\cite{ABN-2022}, we will provide a sufficient condition for a projective code $\cC(\Omega)$ to be minimal by investigating some geometrical properties of the projective system $\Omega.$ We will also
exhibit a family of codes which are minimal but do not satisfy condition~\eqref{ab-bound}.

Suppose  $\cC(\Omega)$ is a projective code arising from the projective system $\Omega$, then $\cC(\Omega)$ is minimal if and only if $\Omega$ is a \emph{cutting blocking set with respect to hyperplanes}, i.e.   any hyperplane $H$ of $\PG(\langle\Omega\rangle)$ is spanned by its section with $\Omega$ (i.e.$ \langle H\cap\Omega\rangle=H$); see~\cite{ABN-2022,Alfarano2022,Bonini2021, DGMP2011,Tang2021}.

In this paper we will focus on the linear codes constructed by taking as projective system, the image of an arbitrary point-line geometry $\Gamma$ under a projective embedding. In order to state our main results we will briefly recall what we mean by projective embedding of a point-line geometry, referring the reader to Section~\ref{Sec point-line geometries} for more details.

Let $\Gamma=(\cP,\cL)$ be a point-line geometry and $\PG(V)$ be
a projective space.
An injective map  $\vep:\Gamma\to\PG(V)$ 
from the pointset $\cP$ to the pointset of $\PG(V)$ is
a \emph{projective embedding} of $\Gamma$
if it sends any line of $\Gamma$ onto a projective line of
$\PG(V)$ and $\langle\vep(\cP)\rangle=\PG(V).$
The vector dimension of $V$ is called the \emph{dimension} of $\vep$.

A set $\cH$ of points of $\cP$ is a \emph{geometric hyperplane} of $\Gamma$ if any line of $\Gamma$ either is contained in it or
intersects $\cH$ in just one point.

If $H$ is a projective hyperplane of $\PG(V)$ then $\vep^{-1}(H)$ is always a geometric hyperplane of $\Gamma.$ We say that a geometric hyperplane $\cH$ of $\Gamma$ \emph{arises form the embedding} $\vep$ if there exists a hyperplane $H$ of $\PG(V)$ such that $\cH=\vep^{-1}(H\cap \vep(\cP)).$

 If we take $\Omega:=\vep(\cP)$ as projective system of $\PG(V)$ and consider the projective code $\cC(\Omega)$, relying on the properties of projective embeddings, we have an immediate interpretation of the main parameters of $\cC(\Omega)$: the dimension of $\cC(\Omega)$ is the dimension of the embedding $\vep$, the length of $\cC(\Omega)$ is the cardinality of $\cP$ and the minimum distance of $\cC(\Omega)$ amounts to the difference between the size of $\cP$ and the maximal cardinality of a geometrical hyperplane of $\Gamma$ arising from $\vep.$

In this paper we will investigate some the properties of $\Gamma$ that guarantee the minimality of $\cC(\vep(\Gamma)).$
These are the main results.

\begin{theorem}\label{main thm}
  Let $\Gamma=(\cP,\cL)$ be a point-line geometry and $\vep:\Gamma\to\PG(V)$ be
  a projective embedding of $\Gamma.$
  If for any hyperplane $H$ of $\PG(V)$
  the graph induced by the collinearity graph of $\Gamma$ on $\cP\setminus\vep^{-1}(H)$  is connected, then the
  projective code $\cC(\vep(\cP))$ is minimal.
\end{theorem}

\begin{corollary} \label{main cor}
  \label{c:min}
   If $\Gamma=(\cP,\cL)$ is a point-line geometry whose
   geometric hyperplanes are all maximal then
   the code $\cC(\vep(\Gamma))$ is minimal for \emph{any}
   projective embedding   $\vep$ of $\Gamma$.
\end{corollary}

In the second part of the paper we will examine some remarkable families of point-line geometries for which Theorem~\ref{main thm} and Corollary~\ref{main cor} hold, thus obtaining
\begin{theorem} \label{cor minimality}
  \label{c:minC}
Grassmann codes; Segre codes; line polar Grassmann codes of orthogonal, symplectic, hermitian type;  codes arising from dual polar spaces of orthogonal and symplectic type and codes arising form the point-hyperplane geometry of a projective space are all minimal codes.
\end{theorem}

\begin{remark}
 Projective codes can be defined as those codes whose dual has
  minimum distance at least $3$ or, equivalently, codes where no
  two columns of the generator matrix are proportional.

  Since we are interested in codes arising from projective embeddings
  of geometries, we assume throughout the paper that $\Omega$ is
  a set of $N$ distinct projective points; as such,
  as mentioned before,  $\cC(\Omega)$ is a projective code.

  It is also possible to pose the problem of minimality in the
  more general context of \emph{non-degenerate} codes, i.e. codes
  whose generator matrix does not contain any zero column.
  In this case the set $\Omega$ must be replaced by a multi-set,
  thus assigning to any point a multiplicity, possibly larger than $1$.
  
  The study of the minimality of non-degenerate codes reduces
  to that of projective codes; see Theorem~\ref{ndc}. So,
  apart from the discussion in
  Section~\ref{nd-proj} we shall here limit our
  focus to projective codes and always silently assume
  $\Omega$ to be a set.
\end{remark}

 \subsection{Structure of the paper}
 The structure of the paper is the following. In Section~\ref{prelim}
we recall some basic notions about point-line geometries and their
embeddings; we also fix the notation we shall use throughout the paper.
In Section~\ref{nd-proj} we investigate the link between minimality
for projective and non-degenerate codes.
In Section~\ref{ssmin} 
we point out properties which can be used in order
to prove the minimality of codes arising from embeddings of geometries, by 
considering the link between minimal codewords and properties of geometric hyperplanes.
In Section~\ref{emb-codes} we consider the codes arising from the
Grassmann geometry, the Segre geometry, the polar Grassmann geometries and the point-hyperplane geometry and prove that these codes are minimal (from Theorem~\ref{main thm} and Corollary~\ref{main cor}). For these families of codes we also survey the current state of the art related to the knowledge of their main parameters. In the last section we list
some open problems.

\section{Preliminaries}~\label{prelim}
\subsection{Embeddable point-line geometries}\label{Sec point-line geometries}
We refer to~\cite{Shult1995} for the general theory of point-line geometries. Here, we recall only those properties that will be of use in the framework of projective codes.

A \emph{point-line geometry} is a pair
$\Gamma=(\cP,\cL)$ where $\cP$ is a non-empty set whose elements are called \emph{points}, $\cL$ is a set of subsets of $\cP$ called \emph{lines} and incidence is given by symmetrized inclusion. It is required that any line contains at least two points and for any two distinct points there exists at most one line through them.

Two points $p,q\in\cP$ with $p\neq q$ are \emph{collinear}
if they belong to a common line.
The \emph{collinearity graph} of $\Gamma$ is the graph $G_{\Gamma}$ whose
vertices are the points of $\Gamma$ and such that an edge connects two vertices if and only if corresponding points are collinear.
We say that the geometry $\Gamma$ is \emph{connected} if $G_{\Gamma}$ is a connected graph.

A \emph{subspace} of a point-line geometry $\Gamma$ is a non-empty subset $\cX$ of the point-set $\cP$ of $\Gamma$ such
that, for every line $\ell$ of $\Gamma$, if $|\ell \cap \cX| > 1$ then $\ell \subseteq X$. A proper subspace $\cH$ of $\Gamma$ is
said to be a \emph{geometric hyperplane} of $\Gamma$ (a hyperplane of $\Gamma$< for short) if
for every line $\ell$ of
$\Gamma$, either $\ell \subseteq \cH$ or $|\ell  \cap \cH | = 1$.

By definition, hyperplanes are proper subspaces of $\Gamma$.
However, in general,
they might not necessarily be \emph{maximal} with respect to
inclusion among the set of all the proper subspaces of $\Gamma$.
In particular, a geometric hyperplane might be a proper subset of
another geometric hyperplane.

An important characterization of \emph{maximal hyperplanes}
is the following, see~\cite[Lemma 4.1.1]{Shult2011}.
\begin{prop}
  \label{cmhp}
  A geometric hyperplane $\cH$ of a point-line geometry $\Gamma$ is
  a maximal subspace if and only if the collinearity
  graph induced by $G_{\Gamma}$
  on $\cP\setminus \cH$ is connected.
\end{prop}


Given  a projective space $\PG(V)$, a \emph{projective embedding $\vep\colon \Gamma \rightarrow \PG(V)$} is an injective mapping $\vep$ from the point-set  $\cP$ of $\Gamma$  to the point-set of a projective space $\PG(V)$  such that $\vep(\cP)$ spans $\PG(V)$ and it maps lines of $\Gamma$ onto projective lines of $\PG(V)$. The \emph{dimension of} $\vep$ is the (vector) dimension of $V$.


Given a projective embedding $\vep\colon \Gamma \rightarrow \PG(V)$
of $\Gamma$ and a projective hyperplane $H$ of $\PG(V)$, the point set $\cH:=\vep^{-1}(H)\subseteq\Gamma$ is clearly  a geometric hyperplane of $\Gamma$ and $\vep( \vep^{-1}(H)) = H\cap  \vep(\cP)$.

We say that a geometric hyperplane $\cH$ of $\Gamma$
\emph{arises from $\vep$} if
$\cH=\vep^{-1}(H)\subseteq\Gamma$ for some hyperplane $H$ of $\PG(V)$.
Note that, in general, there might exist hyperplanes of $\Gamma$ which do not arise from a specific embedding or even from any embedding at all.
In particular, when $\vep(\cH)$ spans $\PG(V)$ then $\cH$
certainly does not arise from $\vep$.

Given a point-line geometry  $\Gamma$  and a projective embedding $\vep\colon \Gamma\rightarrow \PG(V)$ of $\Gamma$, it is now clear from the  definition  that the set
\[\vep(\cP)=\{\vep(p)\colon p\in \cP\}\]
 is a projective system of $\PG(V).$ Put $\Omega:=\vep(\cP)$ and consider the linear code $\cC(\Omega)$ arising from it. Then the length of $\cC(\Omega)$ is the cardinality $|\cP|$ of the pointset of $\Gamma$, the dimension of $\cC(\Omega)$ is precisely the dimension $\dim(\vep)$ of the embedding $\vep$ and the minimal distance is the difference between $|\cP|$ and the maximal cardinality of a geometric hyperplane of $\Gamma$ arising from $\vep.$

\subsection{Cutting blocking sets}\label{sec 2.2}
Given a $(n+1)$-dimensional vector space $V=V(n+1,q)$ over the finite field $\FF_q,$ we refer to a
$n$-dimensional projective space defined over $V$ with the notation
$\PG(n,q)=\PG(V). $ 
When we need to distinguish between a non-zero vector $x$ of $V$ and the point
of $\PG(V)$ represented by it, we denote the latter by $[x]$. We extend
this convention to subsets of $V$. If $X \subseteq V \setminus \{0\}$ then $[X] := \{[x] | x \in X\}$. The
same conventions will be adopted for vectors and subsets of the dual space $V^*$  of $V.$ In
particular, if $\xi \in V^* \setminus \{0\}$ then $[\xi]$ is the point of $\PG(V^*)$ which corresponds
to the hyperplane $[\ker(\xi)]$ of $\PG(V )$. We shall freely take $[\xi]$ as a name for $[\ker(\xi)]$ and put $[0]:=\PG(V)$.

The notion of \emph{strong blocking set} has been introduced
in~\cite{DGMP2011}.

\begin{definition}  \label{cutting bs}
A set $X$ of points of $\PG(n,q)$ is a
\emph{cutting blocking set (with respect to hyperplanes)} or
a \emph{strong cutting set with respect to hyperplanes}
if it meets each hyperplane in a spanning set.
\end{definition}

Recently, several bounds on the minimal size of
cutting blocking sets and also some constructions of small minimal blocking sets with respect to hyperplanes have been found (see e.g. \cite{Alfarano2022, Alon2024, Bishnoi2024,Scotti2024}).


\subsection{Minimal codes}
Given a $[N,k,d]_q$-linear code $\cC$, regarded as a subspace of $\FF_q^N$, and a codeword $c=(c_i)_1^N \in\cC$, the \emph{support} of $c$ is defined as the set
$\mathrm{supp}(c):=\{ i\in\{1,\dots,N\}: c_i\neq 0\}$.
The \emph{weight} of $c$ is
$wt(c):=|\mathrm{supp}(c)|$. It is well known that for linear
codes the minimum Hamming distance $d$ between any two distinct codewords is the same as its minimum weight.

Following Massey~\cite{Massey}, minimal codewords  and minimal codes are defined as follows.
\begin{definition}
  \label{minc}
  A codeword $c\in\cC$ is \emph{minimal} if
  \[ \forall c'\in\cC:\mathrm{supp}(c')\subseteq\mathrm{supp}(c)
    \Rightarrow\exists\lambda\in\FF_q: c'=\lambda c. \]
  A code is \emph{minimal} if all its codewords are minimal.
\end{definition}
If a code is minimal, then all its codewords are determined (up to
a multiplicative factor) by their support. As the following remark shows,
the converse is not true in general.
\begin{remark}
  Let $\Omega$ be the projective system given by the points of an elliptic
  quadric $Q^-(3,q)$ in $\PG(3,q)$. Then, a hyperplane of $\PG(3,q)$
  either is the unique tangent plane to $Q^-(3,q)$ at a point $P$ or
  it is spanned by its intersection with $Q^-(3,q)$. So the support of
  a nonzero codeword either has cardinality $(q^2+1)-1$ or $(q^2+1)-(q+1)$,
  and the support of the codewords corresponding to non-tangent hyperplanes
  are contained in all the supports of the codewords corresponding to
  hyperplanes tangent at their points.
\end{remark}
The codes where all words are determined (up to scalar multiples) by their
support are called \emph{quasi-minimal}; see~\cite{Cohen2013}.

The minimality of a code is closely
related to some deep structural properties of the code,
 as well as to the construction of efficient
 decoding algorithms, see~\cite{Ashikhmin1998}; thus it is a relevant
 property to investigate.

Proving the minimality of a linear code is reputed to be
a hard problem, as it would require at least to have
some insight on the weights of the code, itself a NP-complete
problem.
For this reason, in recent years, structural conditions guaranteeing
the minimality of a code or constructions yielding
minimal codes have been extensively looked after.

As mentioned in the Introduction (see~\eqref{ab-bound}), Ashikhmin and Barg in \cite{Ashikhmin1998} determined
a sufficient condition for a code $\cC$ to be minimal; this
condition depends only on the minimum and maximum weight of $\cC$.
We remark that while the condition~\eqref{ab-bound} is sufficient
for minimality, there exists many families of codes which are
minimal but do not satisfy it (see~\cite{Pasalic2021} and~\cite{Chen2025}) or for which we do not have sufficient insight on
the maximum weight (e.g. the Grassmann codes for $k>3$) as to
be able to check if it is satisfied.
 We will provide some examples of such families in Remark~\ref{esempio sec4} of Section~\ref{emb-codes}.

An important tool for determining the minimality of a code
is the following characterization by~\cite{ABN-2022}
(see also~\cite{Tang2021}), relating
minimal codes and cutting sets with respect
to hyperplanes. We shall make use of this property extensively.

\begin{prop}\label{equiv}
  A projective code $\cC:=\cC(\Omega)$ is minimal if and only if
  $\Omega$ is a cutting set with respect to hyperplanes of $\PG(\langle\Omega\rangle)$.
\end{prop}

\section{Non-degenerate codes}
\label{nd-proj}
Recall that a code $\cC$ is non-degenerate if its generator matrix does
not contain any zero column; this is equivalent to say that
the minimum distance of
the dual of $\cC$ is at least $2$.
\begin{definition}
  Let $\cC$ be a non-degenerate code with generator matrix
  $G$.
  A \emph{projectivized code} $\cC^P$ corresponding to $\cC$ is
  a projective code
  having as generator matrix the submatrix $G^P$ of $G$
  obtained by removing from $G$ all its proportional columns
  but one.
\end{definition}
Projectivized codes are uniquely defined up to equivalence.
If $\cC$ is projective, then $\cC^P$ is equivalent to
$\cC$; otherwise, $\cC^P$ is a shorter code.  Let $n$ and $n^P$ be the length of $\cC$ and
  $\cC^P$ respectively.
When $n^P<n$, we can always suppose that, up to code equivalence,
the generator matrix $G$ of $\cC$ can be written as
$G=( G^P| S )$ where $G^P$ is a generator matrix for
the projectivized code $\cC^P$.
More explicitly,
suppose $G^P=(G_1,\dots,G_{n^P})$ and $S=(S_1,\dots,S_{n-n^P})$,
where the $G_i,\, 1\leq i\leq n^P$ and $S_j, 1\leq j\leq n-n^P$ are respectively the columns of $G^P$ and $S$.

The codes $\cC$ and $\cC^P$ have the same dimension; so they also
have the same number of words.
Any codeword $c=(c^P, \overline{c})\in \cC$ is associated to a unique codeword $c^P\in \cC^P$, obtained by puncturing the last $n-n^P$ coordinates from $c$.    Conversely, for any codeword $c^P\in \cC^P$ there exists a unique codeword $c\in \cC$  associated to $c^P$, having as first $n_P$ components
the components of $c$.
More in detail, let $m$ be the message such that $c^P=mG^P.$
Since $G=( G^P| S )$ we can associate to $c^P\in \cC^P$ the codeword $c=mG=m ( G^P| S )=( mG^P| nS )=(c^P,mS)\in \cC$ obtained by concatenating $c^P$ with $mS.$ This
realizes a bijective correspondence between the codewords of
$\cC$ and those of its projectivized code $\cC^P.$
Clearly, $\mathrm{supp}(c^P)\subseteq\mathrm{supp}(c).$

The following theorem relates the minimality of $\cC^P$ and $\cC$.
\begin{theorem}
  \label{ndc}
  Let $\cC$ be a non-degenerate code and $\cC^P$ its projectivization.
  The linear code $\cC$ is minimal if and only if
 $\cC^P$ is minimal.
\end{theorem}
\begin{proof}
  Suppose $\cC^P$ to be minimal and take two codewords $c=(c^P|s),
  c'=({c'}^P|s')$ of $\cC$ which are not proportional.
  Since $c^P$ is not proportional to ${c'}^P$,
  there exist $1\leq i,j\leq n_P$ such that
  $i\in\mathrm{supp}({c'}^P)\setminus\mathrm{supp}(c^P)$
  and $j\in\mathrm{supp}(c^P)\setminus\mathrm{supp}({c'}^P)$.
  Since $\mathrm{supp}(c)=\mathrm{supp}(c^P)\cup
  \{ n_P+r\colon r\in\mathrm{supp}(s) \}$
  and
   $\mathrm{supp}(c)=\mathrm{supp}({c'}^P)\cup
   \{ n_P+r\colon r\in\mathrm{supp}(s') \}$,
   we have $i\in\mathrm{supp}({c'})\setminus\mathrm{supp}(c)$
   and $j\in\mathrm{supp}({c})\setminus\mathrm{supp}(c')$.
   It follows that $\cC$ must be a minimal
   code.

   Conversely, suppose $\cC$ to be a minimal code.
   Assume by contradiction that
   $\cC^P$ is not minimal. Then there exist two codewords
   $c=(c^P|s) ,c'=({c'}^P|s')\in\cC$ such that, say
   $\mathrm{supp}(c^P)\subseteq\mathrm{supp}({c'}^P)$
   and $\mathrm{c}^P$ is not a scalar multiple of ${c'}^P$.

   Since $\cC^P$ is projective, no two columns of $G^P$ are proportional;
also, since $\cC$ is non-degenerate,
for any column $S_j$ of $S$ there exists,  by construction, a unique
column $G_i$ of $G^P$ such that $S_j$ is a non-zero multiple of $G_i$.
Denote by $\psi:\{1,\dots,n-n_P\}\to\{1,\dots,n_P\}$ the function
mapping the index $j$ of a column in $S$ to the index $i$ of
the corresponding column in $G$.

Observe that
$\mathrm{supp}(c)=\mathrm{supp}(c^P)\cup\{ n_P+j :
c^P_{\psi(j)}\neq 0 \}$ and
$\mathrm{supp}(c')=\mathrm{supp}({c'}^P)\cup\{ n_P+j :
{c'}^P_{\psi(j)}\neq 0 \}$.
Since $\mathrm{supp}(c^P)\subseteq\mathrm{supp}({c'}^P)$,
we have that $c^P_{\psi(j)}\neq0$ implies ${c'}^P_{\psi(j)}\neq0$,
so $\{ n_P+j :c^P_{\psi(j)}\neq 0 \}\subseteq\{ n_P+j :
{c'}^P_{\psi(j)}\neq 0 \}$.
Thus, $\mathrm{supp}(c)\subseteq
\mathrm{supp}(c')$ and $c$ and $c'$ are linearly independent.
This contradicts the minimality of $\cC$.
The proof follows.
\end{proof}

\begin{remark}
  If $\cC$ is degenerate, then it is equivalent to a code
  whose generator matrix is of the form $(G|\mathbf{0})$
  where $\mathbf{0}$ is a null matrix and $G$ is the generator
  matrix of a non-degenerate code. In particular, since
  no word has support on the positions of $\mathbf{0}$,
  we have that $\cC$ is minimal if and only if the
  projectivization $\cC^P$ of the code with generator matrix
  $G$ is minimal.
  \end{remark}

  \begin{remark}
    \label{abnon}
    Condition~\eqref{ab-bound} is a widely used tool in order to
    test for the minimality of a code and most of the known minimal
    codes satisfy it.
    On the other hand, it is easy to construct codes which are minimal
    but fail to fulfill it.

    To this aim, let $\cC$ be a $[n,k]$-minimal code with
    minimum and maximum weight respectively $w_{\min}$ and
    $w_{\max}$ and $k>1$.
    Let
    $G$ be a generator matrix of $\cC$ such that  
    the first row of $G$ is a word $c_{\min}$ of minimum weight
    and the last row
    of $G$ is a word $c_{\max}$ of maximum weight $w_{\max}.$

    We can now take the extended code
  $\cC^{(t)}$ with generator matrix $G^{(t)}:=(G|U^{(t)})$ where $U^{(t)}$ is a
    $k\times t$ matrix which is zero everywhere but in the last row,
    which consists of the all $1$-vector.

    Observe that all linear combinations of the rows of $G^{(t)}$ where the coefficient
    of $c_{\max}$ is $0$ yield codewords of $\cC^{(t)}$
   which have the same weight as codewords of $\cC$.
    When the coefficient of $c_{\max}$ is non-zero, then we get a codeword of $\cC^{(t)}$
    which has weight $w+t$ where $w$ is the weight of a codeword of $\cC$.
    So,  the minimum weight $w_{\min}$ of $\cC$ and  the
    minimum weight $w_{\min}^{(t)}$ of $\cC^{(t)}$ are the
    same while the maximum weight of $\cC^{(t)}$ is $w_{\max}^{(t)}=w_{\max}+t$.
    If we choose the parameter $t$ so that
    \[ t>\left(\frac{q}{q-1}w_{\min}\right)w_{\min}-w_{\max}, \]
    the code $\cC^{(t)}$ does not satisfy anymore the bound~\eqref{ab-bound}.

    Using the same argument as above we see that any two non-proportional
    codewords of $\cC^{(t)}$ must have supports which are not contained in each other;
    so $\cC^{(t)}$ is again minimal.

    In general, if $t>1$ the code $\cC^{(t)}$ is not going to be projective, even if $\cC$ is so.
  \end{remark}
  
  Any extension of a minimal code is minimal; however, puncturing a minimal
  code might yield a code which is not minimal anymore. A minimal code
  whose puncturings are not minimal for any choice of coordinates is
  called \emph{reduced}~\cite{ABN-2022}. As pointed out in~\cite{ABN-2022},
  a reduced minimal code is always projective (the converse
  is obviously false).

\section{Minimal codes and maximal hyperplanes}\label{ssmin}
In this section we shall prove our main results. The key statement is the following proposition.
\begin{prop}
  \label{ptv}
  Let $\Gamma$ be a point-line geometry and $\vep:\Gamma\to\PG(V)$
  be a projective embedding of $\Gamma$. Denote by $\Omega$ the
  image of $\vep$.
  If all the geometric hyperplanes of $\Gamma$ arising from $\vep$ are maximal (as
  subspaces of $\Gamma$) then $\cC(\Omega)$ is a minimal code.
\end{prop}
\begin{proof}
  By contradiction,
  suppose that $\cC(\Omega)$ is not a minimal code. By Proposition~\ref{equiv}, there is at least one
  hyperplane $[\varphi]$  of $\PG(\langle\Omega\rangle)$ such that
  $T:=\langle\Omega\cap [\varphi]\rangle\subset [\varphi]$ with $T\neq[\varphi]$.
  Then ${\mathcal T}:=\vep^{-1}(T)$ is a geometric hyperplane of $\Gamma$.
  Indeed, let $\ell $ be  a line of $\Gamma$, hence $\vep(\ell)$ is a projective line of $\PG(V)$ and either  $\vep(\ell)\cap [\varphi]$ is the  line $\vep(\ell)\subseteq \Omega\cap [\varphi]\subseteq T$ or $\vep(\ell)\cap [\varphi]$ is a (projective) point $p\in \Omega\cap [\varphi]\subseteq T$ because $[\varphi]$ is a hyperplane of $\PG(V).$
  Hence, $\ell\cap \vep^{-1}(T)$ either coincides with the line $\ell$ or it is the point $\vep^{-1}(p)$ of $\Gamma.$

  Take now $[u]\in\Omega\setminus T$. Since $T$ is not a hyperplane of
  $\PG(\langle\Omega\rangle)$ there exists at least an hyperplane $Y$
  of $\PG(\langle\Omega\rangle)$ containing both $T$ and $u$.
  Then ${\mathcal Y}:=\vep^{-1}(Y)$ is a geometric hyperplane of
  $\Gamma$ and ${\mathcal T}\subset{\mathcal Y}$ with ${\mathcal T}\neq{\mathcal Y}$ since $\vep^{-1}(u)\in{\mathcal Y}\setminus{\mathcal T}$.
  This contradicts the hypothesis
  that all hyperplanes of $\Gamma$ arising from $\vep$ are maximal subspaces.
\end{proof}
Theorem~\ref{main thm} follows from Proposition~\ref{cmhp} and Proposition~\ref{ptv}.
Corollary~\ref{main cor} follows from Proposition~\ref{ptv}.

\begin{remark}
  Since the only hyperplanes relevant for determining the minimality
  of a code $\cC(\Omega)$ with $\Omega=\vep(\cP)$ are those arising from the embedding
  $\vep$,
  the hypothesis of Proposition~\ref{ptv} can be slightly weakened, by
  just requiring that the hyperplanes of $\Gamma$ which arise
  from $\vep$ are maximal in the family of all the geometric
  hyperplanes arising from $\vep$. However, this latter
  condition is the same to say that for any hyperplane $H$
  of $\PG(V)$ we have $\dim(\langle H\cap\Omega\rangle)=\dim(V)-1$
  which is just a restatement of the fact that $\Omega$ is
  a cutting set for the hyperplanes.
\end{remark}




\section{Some remarkable families of codes}~\label{emb-codes}
In this section we shall discuss some remarkable families of codes
arising from projective embeddings of point-line geometries.
Since the embeddings being considered are \emph{homogeneous} (see e.g.~\cite{Cardinali2018a}), all these
codes have a large automorphism group which is actually transitive on
the components of the codewords.

\subsection{Grassmann codes}\label{Grassmann codes}
Grassmann codes have been introduced in \cite{R1,R2,R3} as generalizations
of first order Reed--Muller codes and
they have been extensively investigated, both for their geometry
and for their properties.
We briefly recall their construction.

Suppose $V$ is a $n$-dimensional vector space over a finite field $\FF_q$ and for $k=1,\dots, n-1$, denote by $\cG_{n,k}=(\cP,\cL)$ the $k$-Grassmann geometry of the projective space $\PG(V)$, that is the point-line
geometry whose pointset is the set of the $k$-dimensional subspaces of $V:$
\[\cP=\{X\colon \dim(X)=k\}\] and whose
line set is
\[\cL=\{\ell_{W,T}\colon \dim(W)=k-1 \text{ and } \dim(T)=k+1\}\]
 where $\ell_{W,T}:=\{ X: W\leq X\leq T, \dim(X)=k \}.$

 The geometry $\cG_{n,k}$ affords the so-called
 Pl\"ucker of
 Grassmann embedding $\vep_{k}^{gr}$ which maps any arbitrary point of $\cG_{n,k}$, that is a
$k$--dimensional subspace $\langle v_1,v_2,\dots,v_k\rangle$ of
$V$,  to the point $[v_1\wedge v_2\wedge\dots\wedge v_k]$ of $\PG(\bigwedge^kV)$, where $\bigwedge^k V$  is  the $k$-th  exterior power of $V$:
\begin{equation}\label{Grassmann embedding}
\begin{array}{ll}
\vep_{k}^{gr}:&\cG_{n,k}\to\PG(\bigwedge^kV) \\
 &\vep_{k}^{gr}(\langle v_1,v_2,\dots,v_k\rangle)= [v_1\wedge v_2\wedge\dots\wedge v_k]\end{array}
\end{equation}

The  Grassmann embedding  is a projective embedding, as defined in Section~\ref{Sec point-line geometries}, and the set $\vep_{k}^{gr}(\cP)$ can be regarded as a projective system of $\PG(\bigwedge^kV).$ Thus we can consider the linear code $\cC_{n,k}:=\cC(\vep_{k}^{gr}(\cP))$ arising from $\vep_{k}^{gr}(\cP)$, called \emph{Grassmann code}.

It is known from the literature~\cite[Lemma 2.1]{Bart} that all geometric hyperplanes of the geometry $\cG_{n,k}$ arise from the  Grassmann  embedding and that every hyperplane of $\cG_{n,k}$ is maximal.
So, the code $\cC_{n,k}$ is minimal.


So, combining the result of~\cite{Nogin1996} for the parameters and Proposition~\ref{ptv} and ~\cite{Bart} for the minimality, we have
\begin{theorem}\label{min1}
A Grassmann code is a $[N,K,d]_q$- minimal code with
\[ N=\genfrac[]{0pt}0{n}{k}_q=
  \frac{(q^n-1)(q^n-q)\dots(q^n-q^{k-1})}{(q^k-1)(q^k-q)\dots(q^k-q^{k-1})},
\]
\[ K={n\choose k},\qquad d=q^{(n-k)k}. \]
\end{theorem}
Minimum weight codewords have  been characterized by Nogin~\cite{Nogin1996} as those corresponding to completely
decomposable elements of $\bigwedge^{n-k}V\cong(\bigwedge^k V)^*$;
see also~\cite{Ghorpade2009}.
\begin{theorem}[\cite{Ghorpade2009}]
  A linear functional $f\in(\bigwedge^kV)^*\cong
  (\bigwedge^{n-k}V)$ determines a minimum weight codeword of
  $\cC_{n,k}$ if and only if there exist $n-k$ linearly independent
  vectors $v_1,\dots,v_{n-k}\in V$ such that
  $\omega_0:=(v_1\wedge\dots\wedge v_{n-k})$ and
  $f(x):=\omega_0\wedge x$.
  In particular, the number of minimum weight codewords of
  $\cC_{n,k}$ is \[(q-1)\genfrac[]{0pt}0{n}{n-k}_q=
  (q-1)\genfrac[]{0pt}0{n}{k}_q. \]
\end{theorem}

For any $n$ and $k=2$ the complete weight spectrum of Grassmann line codes  $\cC_{n,2}$ has been computed in~\cite{Nogin1996}.
The number $A_i$ of codewords of weight $i$ of $\cC_{n,2}$  is
\[ A_i=\begin{cases}
  \beta(n,2r) & \text{ if } i=q^{2(n-r-1)}\frac{q^{2r}-1}{q^2-1}
                \text{ and } 0\leq r\leq\lfloor\frac{n}{2}\rfloor \\
  0 & \text{otherwise}
\end{cases} \]
where
\[ \beta(n,2r):=\frac{(q^n-1)(q^{n-1}-1)\dots(q^{n-2r+1}-1)}{
    (q^{2r}-1)(q^{2r-2}-1)\dots(q^2-1)}q^{r(r-1)}. \]
Many (but not all) generalized Hamming weights are also known for $k=2$; see~\cite{Ghorpade2009}.

The computation of the complete weight enumerator of $\cC_{n,k}$
for arbitrary $n$ and $k>2$ is an open problem.

For any $n$ and any $k$, the second minimum weight of $\cC_{n,k}$ has been
determined in~\cite{Datta2025} as
\[ q^{k(n-k)}+q^{k(n-k)-2}. \]
For any $n$ and $k=n-1$ the Grassmann geometry $\cG_{n,n-1}$ is just
the dual projective space $\PG(n-1,q).$
For any $n$ and $k=n-2$ the Grassmann geometry $\cG_{n,n-2}$ is isomorphic to the line Grassmann geometry
$\cG_{n,2}$. So, these last two cases are covered.

For $(n,k)=(6,3)$ the weight spectrum for $\cC_{6,3}$ has been determined
first for $q=2$ in~\cite{R4} and then in~\cite{Nogin1997}
for general $q$.
The nonzero weight distribution is outlined in the following table
\[ \begin{array}{l|c}
  \text{Weight} & \text{Number of codewords} \\
  \hline
  q^9                & (q^5-1)(q^3-1)(q^2+1) \\
  q^9+q^7            & (q^6-1)(q^5-1)(q^2+q+1)q^2 \\
  q^9+q^7+q^6-q^4    & q^9(q^5-1)(q^3+1)(q^2+1)(q-1)/2 \\
  q^9+q^7+q^6        & (q^6-1)(q^5-1)q^4(q^4-1) \\
  q^9+q^7+q^6+q^4    & q^9(q^5-1)(q^3-1)(q^2-1)(q-1)/2 \\
\end{array}. \]

For $(n,k)=(7,3)$ the weight spectrum for $\cC_{7,3}$ has been addressed in ~\cite{Kaipa2013}; in particular the associated
code has the following $10$ distinct nonzero weights:
\[ \begin{array}{l|c}
  \text{Weight} & \text{Number of codewords} \\
  \hline
  q^{12} & (q^7-1)(q^5-1)(q^2-q+1)/(q-1) \\
 q^{12}+q^{10} & q^2(q^7-1)(q^5-1)(q^4+q^2+1)(q^2+q+1) \\
  q^{12}+q^{10}+q^9-q^7 & \frac{1}{2}q^9(q^7-1)(q^5-1)(q^3+1)(q^2+1) \\
  q^{12}+q^{10}+q^9 & q^4(q^7-1)(q^6-1)(q^5-1)(q^4-1)/(q-1) \\
  q^{12}+q^{10}+q^9+q^7 & \frac{1}{2}q^9(q^7-1)(q^5-1)(q^3-1)(q^2-1) \\
  q^{12}+q^{10}+q^9+q^8-q^7 & \frac{1}{2}q^9(q^7-1)(q^6-1)(q^5-1)(q^2+q+1)(q^2+1) \\
  q^{12}+q^{10}+q^8 & q^6(q^7-1)(q^5-1)(q^3-1) \\
  q^{12}+q^{10}+q^9+q^8 & q^{11}(q^7-1)(q^6-1)(q^5-1)(q^3-1)(q^2+1) \\
                      & +q^6(q^7-1)(q^6-1)(q^5-1)(q^4-1) \\
  q^{12}+q^{10}+q^9+q^8+q^6 & q^{15}(q^7-1)(q^5-1)(q^4-1)(q^3-1)(q-1) \\
  q^{12}+q^{10}+q^9+q^8+q^7 & \frac{1}{2}q^9(q^7-1)(q^6-1)(q^5-1)(q^3-1)(q-1)
\end{array} \]
It can be easily verified that, in all known cases, Grassmann codes satisfy the bound~\eqref{ab-bound}.
As mentioned before,
for other values of $n$ and $k$, the problem of determining the complete weight spectrum
of  $\cC_{n,k}$ is open as well as that of determining its maximum distance, as it would require at least a rough
classification of the possible $k$-multilinear alternating forms
over an arbitrary vector space $V(n,q)$.

We conclude by observing that
Grassmann codes are also of interest since, by Chow's theorem, they
inherit the automorphism group $\PGL(V)$ of $\PG(V)$ as a
large automorphism group~\cite{Ghorpade2013} acting  transitively
on the components of the codewords; this leads also to efficient
error correction and decoding algorithms; see~\cite{Beelen2021}.


\subsection{Polar Grassmann codes}
Polar Grassmann codes have been introduced in a  series of papers (~\cite{Cardinali2013,  Cardinali2016b, Cardinali2018, Cardinali2018c,
  Cardinali2016a}) and  can be regarded as subcodes obtained
by puncturing  the Grassmann codes  described
in Section~\ref{Grassmann codes}.
We recall their construction in general.

Suppose $V$ is a vector space with $\dim(V)=t$
and $\varphi$ a non-degenerate quadratic, alternating or hermitian  form defined over $V$ with Witt index $n.$
For $k=1,\dots, n-1$, the \emph{polar grassmannian} $\Delta_{n,k}=(\overline{\cP}, \overline{\cL})$ is the point-line geometry where the pointset $\overline{\cP}$ is the set of the $k$-dimensional subspaces of $V$ totally isotropic with respect to $\varphi$ and the lineset is the set $\overline{\cL}=\{\ell_{W,T}\colon \dim(W)=k-1 \text{ and } \dim(T)=k+1\}$
where $\ell_{W,T}:=\{ X: W\leq X\leq T, \dim(X)=k \}$ and $T$ is totally isotropic with respect to $\varphi.$
For $k=n$ the polar grassmannian $\Delta_{n,n}$ is called
\emph{dual polar space}. Its points are the set of maximal
totally isotropic subspaces with respect to $\varphi$ (also called
\emph{generators})
while its lines correspond to the set of all generators containing
a fixed $(n-1)$-dimensional totally isotropic subspace.

It is clear that the polar grassmannian $\Delta_{n,k}$ for $k<n$ is a subgeometry of the $k$-Grassmann geometry of $\PG(V).$
If $k=n$ the geometry $\Delta_{n,n}$ is not necessarily a subgeometry
of the $n$-Grassmann geometry of $\PG(V)$ but, in any case, its
points are elements of $\cG_{t,n}$ and, as such, we can still consider
their image under the Pl\"ucker embedding and study the corresponding
code.

If we consider the restriction $\overline{\vep}_{k}^{gr}$ to $\Delta_{n,k}$ of the  Pl\"ucker embedding $\vep^{gr}_k$ defined in~\eqref{Grassmann embedding} we have an embedding of the polar grassmannian $\Delta_{n,k}$ called again  Pl\"ucker embedding.
The set
\[\overline{\vep}^{gr}_k(\overline{\cP}):=\{\overline{\vep}^{gr}_k(p)\colon p\in \overline{\cP}\}\]
 is a projective system of the projective space $\langle \overline{\vep}^{gr}_k(\Delta_{n,k})\rangle\subseteq \PG(\bigwedge^k V).$

 A \emph{polar Grassmann code} $\cC(\Delta_{n,k})$ is the linear code arising from $\overline{\vep}^{gr}_k(\overline{\cP})$.
 Its length is the size of the pointset $\overline{\cP}$ of $\Delta_{n,k}$  and its dimension is the dimension $\dim(\overline{\vep}^{gr}_k)$ of the  Pl\"ucker embedding $\overline{\vep}^{gr}_k$.
 As the  Pl\"ucker embedding is homogeneous when $k<n$, all these codes inherit
 the automorphism group of the polar space as a monomial automorphism
 group; see~\cite{Cardinali2018a} for more details. In the case $k=n$ the Pl\"ucker
 embedding might not be projective.
 In particular, the symplectic, orthogonal and hermitian Grassmann
 codes have as automorphism group $Sp(t,q)$, $O(t,q)$ and $U(t,q^2)$.
For $k=2$, we talk about \emph{line polar grassmannians} and
 the  Pl\"ucker embedding of $\Delta_{n,k}$ is
 hosted in $\PG(\bigwedge^2V)$. Since the codewords of
 $\cC(\Delta_{n,k})$ correspond to linear functionals over
 $\bigwedge^2V$, there is a natural isomorphism between
 $(\bigwedge^2V)^*$ and the space of all alternating bilinear
 forms $\varphi$ over $V$; see e.g.~\cite{Bourbaki1998}.
 We shall always assume that a basis $(e_1,\dots,e_{t})$;
 the dual basis is denoted by $(e^1,\dots,e^t)$.
 Then $(e_i\wedge e_j)_{1\leq j<j\leq t}$ is a basis of
 $\bigwedge^2V$ while $(e^i\wedge e^j)_{1\leq i<j\leq t}$
 is a basis of $(\bigwedge^2V)^*$.

If $\varphi$ is a non-degenerate alternating form, the polar Grassmann code  arising from it is called \emph{symplectic Grassmann code};
if $\varphi$ is a non-degenerate orthogonal form, the code arising from it is called \emph{orthogonal Grassmann code } and if $\varphi$ is a non-degenerate hermitian form, the code arising from it is called
\emph{ hermitian  Grassmann code}.

The only known result on the connectedness of the complement of a geometric hyperplane of $\Delta_{n,k}$ for $k<n$ is proved in~\cite[Theorem 2.1]{Kasikova} and holds for $k=2$  and for polar Grassmann geometries where all lines have at least $4$ points. So, applying~\cite[Theorem 2.1]{Kasikova} and Corollary~\ref{main cor} we have
 \begin{prop}\label{min2}
   Line symplectic Grassmann codes, line orthogonal Grassmann codes and line hermitian Grassmann codes  defined by line polar grassmannians 
   over a finite field $\FF_q\neq\FF_2$  are minimal codes.
 \end{prop}
 For $k=n$, the complement of a geometric hyperplane of a
 classical dual polar space $\Delta_{n,n}$
 is always connected; see~\cite{ShultThas,Cardinali2009,McInroy2010}.
 It is also well known that classical dual polar spaces of
 symplectic and orthogonal type afford projective embeddings.
 So, by Corollary~\ref{main cor} and Proposition~\ref{cmhp}, we have
 \begin{theorem}
   \label{dpp}
   Let $\Delta_{n,n}$ be a dual polar space of either symplectic
   or orthogonal type. Then, for any projective embedding
   $\vep:\Delta_{n,n}\to\PG(V)$ the code $\cC(\vep(\Delta_{n,n}))$
   is minimal.
 \end{theorem}
 \begin{remark}
   The Grassmann embedding of a symplectic dual polar space is projective.
   The Grassmann embedding of an orthogonal dual polar space is not projective;
   however also in this latter case, $\Delta_{n,n}$ affords a projective
   embedding of dimension $2^n$ which is projective: the \emph{spin
     embedding}; see~\cite{Chevalley1954}.
   Theorem~\ref{dpp} in this case applies to the codes
   arising from the image of the spin embedding as
   well as to the quotients of this embedding.
 \end{remark}
 
In the rest of this section we survey the current state of the art art regarding the information available from the literature on the parameters of polar Grassmann codes.

\subsubsection{\sc Symplectic Grassmann codes}
Let $V:=V(2n,q)$ be a $2n$-dimensional vector space equipped
with a non-degenerate bilinear alternating form $\varphi$ with Witt index $n$.

For $k=1,\dots, n$, denote by $\cS_{n,k}$  the symplectic grassmannian induced by $\varphi.$
For $k=n$, $\cS_{n,n}$ is usually called \emph{ symplectic dual polar space of rank $n$}
or the \emph{lagrangian grassmannian}.

The known results on the main parameters of a symplectic $[N,K, d]_q$-Grassmann codes defined by $\varphi$ are summarized in the following.

For $1\leq k\leq n$,
the parameters of $\cC({\cS}_{n,k})$ are
\[N=
\prod_{i=0}^{k-1}\frac{q^{2(n-i)}-1}{q^{i+1}-1},
\qquad K= \binom{2n}{k}-\binom{2n}{k-2}. \]

\begin{theorem}[\cite{mexicans}]
  For $k=n>2$, $d\leq q^{n(n+1)/2}$.
	For $k=n=2$, $d=q^3-q$.
\end{theorem}
\begin{theorem}[\cite{Cardinali2016b}]
  If $k=2$ and $n\geq2$, then {$d=q^{4n-5}-q^{2n-3}$}.
  If $k=n=3$ then  $d=q^6-q^4$.
\end{theorem}
We now describe the minimum weight codewords for $k=2$.
\begin{theorem}[\cite{Cardinali2016b}]
  Let $M$ be the matrix representing the non-degenerate
  alternating form $\varphi$. Then an alternating matrix
  $S$ represents a minimum weight codeword for $\cC(\cS_{n,2})$ if
  and only if $M^{-1}S$ has exactly $2$ eigenspaces of dimension
  respectively $2n-2$ and $2$.
\end{theorem}



\subsubsection{\sc Orthogonal Grassmann codes}
Let $V:=V(2n+1,q)$ be a $(2n+1)$-dimensional vector space and let $\eta:V\to\FF_q$ be a fixed non-degenerate quadratic form over $V$ with Witt index $n.$ For $k=1,\dots, n$, denote by $\cO_{n,k}$ the
\emph{orthogonal grassmannian} associated to $\eta.$ 

The known results on the main parameters of an orthogonal  $[N,K, d]_q$-Grassmann code $\cC(\cO_{n,k})$ are the following

\begin{theorem}\cite{Cardinali2013}
	For $1\leq k< n$,
	the parameters of $\cC(\cO_{n,k})$ are
	\[N=
	\prod_{i=0}^{k-1}\frac{q^{2(n-i)}-1}{q^{i+1}-1},
	\qquad K=\left\{\begin{array}{ll}
		\binom{2n+1}{k} & \text{for $q$ odd} \\
		\binom{2n+1}{k}-\binom{2n+1}{k-2} &
		\text{for $q$ even}\\
	\end{array}\right. \,\,\,\, \]
	\[ d\geq (q+1)(q^{k(n-k)}-1)+1. \]
\end{theorem}

\begin{theorem}\cite{Cardinali2013}
  The following hold:
	\begin{itemize}
		\item The code $\cC(\cO_{2,2})$
		has parameters
		\[N=(q^2+1)(q+1),\qquad
		K=\left\{\begin{array}{ll}
			10 & \text{for $q$ odd} \\
			9  & \text{for $q$ even}
		\end{array}
		\right.,\qquad d=q^2(q-1). \]
		\item The code $\cC(\cO_{3,3})$
		has parameters
		\begin{small}
			\[
			\begin{array}{l l}
				\left.\begin{array}{ll}
					N=(q^3+1)(q^2+1)(q+1),\,\,\,& K=35,\,\,\, \\
					\multicolumn{2}{c}{d=q^2(q-1)(q^3-1)}
				\end{array}\right\}&\text{ for $q$ odd}  \\[.3cm]
				\multicolumn{2}{c}{\text{and}} \\[.2cm]
				\left.\begin{array}{ll}
					N=(q^3+1)(q^2+1)(q+1),\,\,\,& K=28,\,\,\, \\
					\multicolumn{2}{c}{d=q^5(q-1)}
				\end{array}\right\}&\text{ for $q$ even}.\\
			\end{array}\]
		\end{small}
	\end{itemize}
\end{theorem}

  \begin{theorem}[\cite{Cardinali2016a,Cardinali2018c}]
	The code
	$\cC(\cO_{n,2})$ is a $[N,K,d_{min}]$-code  with
	\[ N=\frac{(q^{2n}-1)(q^{2n-2}-1)}{(q-1)(q^2-1)},\,\,\,\, K=\begin{cases} {(2n+1)n} & \text{ if $q$ odd } \\
		{(2n+1)n-1} & \text{ if $q$ even }
	\end{cases}
	\]
	\[{d=q^{4n-5}-q^{3n-4}}. \]
        If $n\neq3$, then the
        second smallest weight is $q^{4n-5}-q^{2n-3}$.
\end{theorem}

The minimum weight codewords for line orthogonal Grassmann codes
can be characterized as follows. For a given reference system of $V$, suppose that the quadratic form $\eta$ can be written, with $x=(x_i)_{i=1}^n$, as
\[ \eta(x):=\sum_{i=1}^n x_{2i-1}x_{2i}+x_{2n+1}^2. \]

Denote by $\cQ$ the parabolic quadric of $\PG(2n,q)$ defined by $\eta.$
\begin{theorem}\cite{Cardinali2016a}
  For $q$ odd and $n>2$
  as well as for $q$ even,
  the codewords of minimum weight all lie on the
  orbit of $e^{1}\wedge e^{2n+1}$ under the action of the orthogonal
  group.
  They all correspond to degenerate alternating
  forms   with radical of vector dimension $2n-1$
  meeting the quadric $\cQ$ in a cone of vertex a point $P$
  projecting a hyperbolic quadric $\cQ^+(2n-3,q)$ of rank $n-1$.

  For $q$ odd and $n=2$ the codewords of minimum weight lie
   either in the orbit of $e^1\wedge e^5$ or in the orbit of
  $e^1\wedge e^2+e^3\wedge e^4$

\end{theorem}
Let now $V:=V(2n,q)$ and suppose that $\eta^+$ is a given non-degenerate quadratic form of $V$
with Witt index $n.$ The set of singular points for $\eta^+$
is a hyperbolic quadric $\cQ^+(2n-1,q)$ of rank $n.$
For $k=1,\dots, n$ denote by $\cO^+_{n,k}$ the orthogonal grassmannian defined by $\eta^+.$
We point out that in the cases considered here the Grassmann map is not a projective
embedding.

The known results on the orthogonal Grassmann code $\cC(\cO^+_{n,k})$ defined by $\eta^+$
 is the following.
\begin{theorem}[\cite{Pinero2025}]
  The code $\cC(\cO^+_{3,3})$ has  parameters
\[N= 2(q+1)(q^2+1)(q^3+1);\,\,\,\, K=\begin{cases}
	20 & \text{ if $q$ is odd } \\
	14 & \text{ if $q$ is even }
	\end{cases};\,\, d=\begin{cases}
	q^3-q^2 & \text{ if $q$ is odd } \\
	q^3 & \text{ if $q$ is even }.
\end{cases}\]
	
The code $\cC(\cO^+_{4,4})$ for $q$ even has parameters
\[N= 2(q+1)(q^2+1)(q^3+1)(q^4+1);\,\,\, K=42;\,\,\, d=q^6.\]
\end{theorem}

\subsubsection{\sc Hermitian Grassmann codes}
Let $V$ be a $m$-dimensional vector space defined over a finite field  of order $q^2.$ Suppose that $V$
is equipped with a non-degenerate hermitian form $\rho$ of Witt index $n$ (hence either $m=2n+1$ or $m=2n$).

For $k=1,\dots,n$, if $m$ is even denote by $\cH_{n,k}^{even}$ the hermitian $k$-grassmannian induced by $\rho$ and if $m$ is odd denote by $\cH_{n,k}^{odd}$ the hermitian $k$-grassmannian induced by $\rho.$ Accordingly, the associated codes arising from the image under the  Pl\"ucker embedding $\vep_{n,k}$ of a hermitian polar grassmannian are denoted by $\cC(\cH_{n,k}^{even})$ and $\cC(\cH_{n,k}^{odd}).$

Note that if $k=2$ and $n>2$ then ${}{\vep_{n,2}}$ maps lines of a line hermitian grassmannian onto projective lines of $\PG(\bigwedge^2 V)$, independently from the parity of $\dim(V)$, i.e. the embedding is
projective. Otherwise, if $n=k=2$ and $m=\dim(V)=5$ then the lines of $\cH_{2,2}^{odd}$ are mapped onto hermitian curves, while if $m=\dim(V)=4$ then lines of $\cH_{2,2}^{even}$ are mapped onto Baer sublines of $\PG(\bigwedge^2 V)$. In the latter case $\vep_{2,2}(\cH_{2,2}^{odd})\cong Q^-(5,q)$ is
contained in a proper subgeometry of $\PG(\bigwedge^2V)$ defined over $\FF_q$.

The known results on the main parameters of a hermitian $[N,K, d]_q$-Grassmann code are the following.


For $1\leq k\leq n$,
the parameters of $\cC({\cH}_{n,k}^{odd})$ and  $\cC({\cH}_{n,k}^{even})$ are
\[N=\frac{\prod_{i={m+1-2k}}^{m}(q^i-(-1)^i)}{\prod_{i=1}^k(q^{2i}-1)},
\qquad K= \binom{m}{k}, \]
where $m=2n$ for $\cC(\cH_{n,k}^{even})$ and $m=2n+1$ for
$\cC(\cH_{n,k}^{odd})$; see~\cite[Theorem 2.19]{GGG} for the value
of $N$.
The following results are from~\cite{Cardinali2018}.
 \begin{theorem}
   \label{main-herm1}
   Suppose $n\geq2$.
   A line hermitian  Grassmann code $\cC(\cH_{n,2}^{even})$  has parameters
   \[ N=\frac{(q^{2n}-1)(q^{{2n}-1}+1)(q^{{2n}-2}-1)(q^{{2n}-3}+1)}{(q^{2}-1)^2(q^{2}+1)};\]
   \[K={2n\choose 2};\,\,\, d=\begin{cases}\displaystyle
    q^{8n-12}-q^{4n-6} & \text{ if ${n}=2,3$ } \\
    q^{{8n}-12} & \text{ if ${n}\geq 4$} \\
  \end{cases}.\]
\end{theorem}

\begin{theorem}
   \label{main-herm}
   A line hermitian  Grassmann code $\cC(\cH_{n,2}^{odd})$  has parameters
   \[ N=\frac{(q^{2n+1}+1)(q^{2n}-1)(q^{2n-1}+1)(q^{2n-2}-1)}{(q^{2}-1)^2(q^{2}+1)};\]
   \[K={2n+1\choose 2};\,\,\,
    d=q^{8n-8}-q^{6n-6}. 
\]
\end{theorem}

\begin{remark}
We point out that the proof of the theorem determining the
minimum distance for line hermitian Grassmann codes presented in~\cite{Cardinali2018} contains an error in the case
of $\cH_{n,2}^{odd}$ for $n=2$ (i.e. in the notation of that
paper for vector dimension  $m=5$).
This mistake
does not change the result asserted by the main theorem
of the paper, but
affects the classification of the codewords of minimum
weight. A corrected version of the proof
is available as an \emph{addendum} to the preprint of
the original paper on \url{https://arxiv.org/abs/1706.10255v4}.
\end{remark}
\begin{theorem}
  For $n=2,3$ the minimum weight codewords of $\cC(\cH_{n,2}^{even})$
  correspond to bilinear alternating forms $\theta$ which are
  permutable with the given hermitian form $\rho$.
  For $n>3$,
  the minimum weight codewords of $\cC(\cH_{n,2}^{even})$
  correspond to bilinear alternating forms $\theta$
  with $\dim(Rad(\theta)) = 2n-2$ and such
  that $[Rad (\theta)]$ meets $\cH_{n,1}^{even}$ in a hermitian cone
  of the form
  $[\Pi_2]\cH_{n-2,1}^+$ where $[\Pi_2]$ is a projective plane.
\end{theorem}
\begin{theorem}
  If $n > 2$ then the minimum weight codewords of
  $\cC(\cH_{n,2}^{odd})$ correspond to bilinear alternating forms $\theta$
  with $\dim(\Rad (\theta)) = 2n - 1$
  and such that $[\Rad (\theta)]$ meets $\cH_{n,1}^{odd}$
  in a hermitian cone of the form $[\ell]\cH_{n-1,1}^{even}$ where $[\ell]$ is a
  projective line.
  For $n=2$, the minimum weight codewords of
  $\cC(\cH_{2,2}^{odd})$
  either are of the
  form described above, or correspond to bilinear alternating
  forms $\theta$ with $\dim(\Rad(\theta))=1$ such that the radical $R$
  of $\theta$ is non-singular
  with respect to the hermitian form $\varphi$,
  $R^{\perp_{\varphi}}$ meets $\cH_{2,1}^{odd}$
  in a non-degenerate hermitian surface $\cH_{2,1}^{even}$ and the restriction
  of the hermitian polarity and the symplectic polarity to
  $R^{\perp_{\rho}}$ commute.
\end{theorem}

\subsection{Segre codes and point-hyperplane codes}
Suppose $\PG(V_1)$ and $\PG(V_2)$ are two given projective spaces
of dimensions respectively $m$ and $n$
over the same field $\FF_q$.
Denote their  pointsets respectively by $\cP_1$ and $\cP_2$ and their linesets by $\cL_1$ and $\cL_2.$
The \emph{Segre geometry (of type $(\dim(\PG(V_1)),\dim(\PG(V_2) ))$)} is defined as the point-line geometry  $\cS_{m,n}=(\cP,\cL)$ having as point-set $\cP$ the cartesian product $\cP_1\times\cP_2$ and as lineset the following set:
\[\cL=\{ \{p_1\}\times\ell_2\colon p_1\in\cP_1,\ell_2\in\cL_2\}
\cup
\{\ell_1\times \{p_2\}\colon \ell_1\in\cL_1, p_2\in\cP_2\};
\]
incidence is given by inclusion.

Let now $\sigma$ be an automorphism of $\FF_q$.
The map
\begin{equation}\label{segre emb}
	\begin{array}{ll}
		\vep_{\sigma}\colon &\cS_{m,n}\to\PG(V_1\otimes V_2)\\
		& \vep_{\sigma}([p],[q])=[p\otimes q^{\sigma}]
	\end{array}
\end{equation}	
mapping the point $([p],[q])\in\cP$ to
the projective point $[p\otimes q^{\sigma}]\in\PG(V_1\otimes V_2)$
is a projective embedding of $\cS_{m,n}$ into
$\PG(V_1\otimes V_2)$.

The image of $\vep_{\sigma}$ does not depend on the
choice of the automorphism $\sigma$ and is called the
\emph{Segre variety} $\SbS_{m,n}$.

So, $\SbS_{m,n}$ is a projective system of $\PG(V_1\otimes V_2)$ and we can consider the linear code arising from it. We call it the \emph{Segre code} $\cC(\SbS_{m,n})$.
The weight enumerator of the Segre codes has been determined in~\cite{BGH}.

\begin{prop}\label{min3}
	The Segre code $\cC(\SbS_{m,n})$ is a $[N,k,d]_q$-minimal code with parameters
	\[ N=\frac{(q^{m+1}-1)(q^{n+1}-1)}{(q-1)^2},
	\qquad k=(m+1)(n+1),\qquad d=q^{m+n}. \]
\end{prop}	
\begin{proof}
	We will first prove that the set $\SbS_{m,n}$ is a cutting set of $\PG(V_1\otimes V_2)$ with respect to hyperplanes.
	Let $H$ be a hyperplane of $\PG(V_1\otimes V_2)$ and consider  $\cH:=\vep_{\sigma}^{-1}(H)$. Clearly, $\cH$ is
	a geometric hyperplane of $\cS_{m,n}$.
	
	Since $H$ is a hyperplane of $\PG(V_1\otimes V_2)$, it has dimension $mn+m+n-1$
	and we get
	\[ \dim(\langle H\cap {\SbS_{m,n}}\rangle)=
	\dim(\langle \vep_{\sigma}(\vep_{\sigma}^{-1}(H))\rangle)
	\leq \dim(H)=mn+m+n-1 \]
	
	It is proved in~\cite[pag. 68]{VanMaldeghem2024},  that the
	span of the image of any geometric hyperplane of $\cS_{m,n}$ under
	the embedding $\vep_{\sigma}$
	has dimension
	either $mn+m+n-1$ or $mn+m+n$.
	
	It follows that $\dim(\langle H\cap {\SbS_{m,n}}\rangle)=mn+m+n-1$ and so $\langle H\cap {\SbS_{m,n}}\rangle=H.$
By Proposition~\ref{equiv}, $\cC(\SbS_{m,n})$ is a minimal code.

The parameters of $\cC(\SbS_{m,n})$ follow from~\cite{BGH}.
\end{proof}

Put $V:=V(n+1,q).$ The \emph{point-hyperplane geometry} of $\PG(V)$ is the point-line geometry $\Gamma=(\cP,\cL)$ whose pointset is
\[\cP=\{(p,H)\colon p\in H,\,\,
p \text{ point of }\PG(V), \, \,H\ \text{ hyperplane of }\PG(V)\}\]
and whose lineset is
\[\cL:=\{\ell_{r,H}\colon r \text{ is a line of }
  \PG(V)\}\cup  \{\ell_{p,S}\colon S \text{ is a subspace of }
  \mathrm{codim}(S)=2\}\]
where $\ell_{r,H}:=\{ (p,H)\in \cP: p\in r \}$ and $\ell_{p,S}:=\{ (p,H)\in \cP : S\subseteq H \}.$

The points $(p,H)$ and $(p',H')$ are collinear in ${\Gamma}$ if and only if $p\in H'$ or $p'\in H$.

The following is known regarding
the geometric hyperplanes of $\Gamma$.
\begin{theorem}[Theorem 1.5, \cite{Pasini24}]
	\label{msub}
	All geometric hyperplanes of ${\Gamma}$ are maximal subspaces.
\end{theorem}
So, by Corollary~\ref{main cor}, we have the following
\begin{prop}\label{min4}
 The code $\cC({\vep}({\Gamma}))$ arising from the projective system ${\vep}({\Gamma})$ is minimal for \emph{any} projective embedding ${\vep}$ of ${\Gamma}.$
\end{prop}

Let $\sigma$ be a non-trivial automorphism of $\FF_q$ and denote by $V^*$ the dual space of $V$. Recall the notation defined at the beginning of Section~\ref{sec 2.2} and consider the map defined as follows, where $W$ is the subspace of $V\otimes V^*$ spanned by the image:
\[{\vep}_{\sigma}:\Gamma\to \PG(W),\,\,([x],[\xi])\mapsto [x \otimes \xi^{\sigma}].\]
The map ${\vep}_{\sigma}$ is a projective embedding of $\Gamma.$
Observe that $V\otimes V^*$ is isomorphic to the vector space
$M_{n+1}(q)$ of all $(n+1)\times(n+1)$ matrices with entries in $\FF_q$.
This vector space is equipped with a natural non-degenerate
bilinear form, the \emph{saturation form} given by
\[ M\cdot N \to \mathrm{Tr}(MN). \]
Define
\begin{equation}
	\label{eS}
	\Lambda_{\sigma}:={\vep}_{\sigma}(\cP) =\{[x\otimes \xi^\sigma]\colon  [x]\in\PG(V), [\xi]\in\PG(V^*)\}.
\end{equation}
Clearly, the set $\Lambda_{\sigma}$ is a projective system of $\PG(W)$ which gives rise to a linear code $\cC(\Lambda_{\sigma}).$

If $\sigma=1$, then  $W$ is exactly the hyperplane of $V\otimes V^*\cong M_n(q)$
given by the equation $ \mathrm{Tr}(X)=0$
and $\dim(\vep_1)= (n+1)^2-1$; $\vep_{1}$ is called the Segre embedding of $\Gamma$ and the related code  $\cC(\Lambda_{1})$ is called the
\emph{point-hyperplane code from the Segre embedding}.

If $\sigma\not=1$ then  $W=V\otimes V^*$ and $\dim(\vep_{\sigma})= (n+1)^2$; $\vep_{\sigma}$ is called the twisted embedding of $\Gamma$ and the related code  $\cC(\Lambda_{\sigma})$ is called the
\emph{point-hyperplane code from the twisted embedding}.

In~\cite{Rodier2003}, independently in \cite{CL-long1}, the following is proved:
\begin{theorem} \label{main thm 1}
The $[N_1,k_1,d_1]$-linear code $\cC(\Lambda_1)$ associated to $\Lambda_1$ is minimal and has parameters \[N_1=\frac{(q^{n+1}-1)(q^n-1)}{(q-1)^2},
\qquad k_1=n^2+2n,\qquad d_1=q^{2n-1} -q^{n-1}.\]
\end{theorem}
When the automorphism $\sigma$ is non-trivial, it is
shown in in~\cite{CL-long2} that
\begin{theorem}
\label{main thm 4}
Suppose now $\sigma\in\Aut(\FF_q)$ with $\sigma\neq 1$.
The code $\cC(\Lambda_{\sigma})$ associated to $\Lambda_{\sigma}$
is minimal and it has
parameters
\begin{itemize}
	\item
	if $n>2$ or $\sigma^2\neq1$:
	\[ N_1=\frac{(q^{n+1}-1)(q^n-1)}{(q-1)^2},
	\qquad k_1=n^2+2n+1,\qquad d_1=q^{2n-1} -q^{n-1}\]
	\item
	if $n=2$ and $\sigma^2=1$:
	\[ N_1=\frac{(q^{n+1}-1)(q^n-1)}{(q-1)^2},
	\qquad k_1=n^2+2n+1,\qquad d_1=q^{3}-\sqrt{q}^3.\]
\end{itemize}
\end{theorem}
A characterization of minimum weight codewords is also available; we refer to~\cite{CL-long1} and~\cite{CL-long2} for the definition of the singular, quasi singular and spread-type hyperplanes of the point-line geometry of $\PG(V).$
\begin{theorem}
\begin{enumerate}
	\item
	If $n>2$ or $\sigma^2\neq1$ or $\sigma=1$,
	then the minimum weight codewords of $\cC(\Lambda_{\sigma})$
	have weight $q^{2n-1}-q^{n-1}$ and  correspond
	to quasi-singular, non-singular hyperplanes of $\Gamma.$ The second lowest weight codewords of $\cC(\Lambda_{\sigma})$ have weight
	$q^{2n-1}$ and  correspond to singular hyperplanes
	of ${\Gamma}$.
	\item
	If $n=2$ and $\sigma^2=1$, $\sigma\neq1$
	then the minimum weight codewords of $\cC(\Lambda_{\sigma})$
	have weight $q^3-\sqrt{q}^3$ and
	correspond
	to the hyperplanes of ${\Gamma}$
        of the form $\mathrm{Tr}(XM)=0$
        associated to matrices $M$ for which there exist three linearly independent vectors
	$\xi_1,\xi_2,\xi_3
	\in V^*$ and
	$\alpha,\beta,\gamma\in\FF_q$ with $N(\alpha)=N(\beta)=N(\gamma)$
	such that
	\[ \xi_1M=\alpha\xi_1^{\sigma},\qquad
	\xi_2M=\beta\xi_2^{\sigma},\qquad
	\xi_3M=\gamma\xi_3^{\sigma}.
	\]
	\item\label{pt4-13} If $\sigma=1$ or both $q$ and $n$ are odd, then the maximum weight codewords of $\cC(\Lambda_{\sigma})$
	have weight $q^{n-1}(q^{n+1}-1)/(q-1)$. Every semi-standard spread type hyperplane of ${\Gamma}$ (in
	particular $\sigma^2=1$ and $n$ is odd) is associated to a maximum weight codeword of $\cC(\Lambda_{\sigma})$.
 \end{enumerate}
\end{theorem}

 \begin{remark} \label{esempio sec4}
 From~\cite{CL-long1} and~\cite{CL-long2} we know that the minimum weight $w_{\min}$ and the maximum weight $w_{\max}$ of the point-hyperplane codes $\cC(\Lambda_{\sigma})$ are
 \[w_{\min}=q^{2n-1}-q^{n-1};\qquad w_{\max}=q^{n-1}(q^{n+1}-1)/(q-1)\]
  so $\frac{w_{\max}}{w_{\min}}=\frac{ q^{n+1}-1}{q^n-1}$
does not satisfy the Ashikhmin and Barg bound~\eqref{ab-bound}, even if  $\cC(\Lambda_{\sigma})$ are minimal codes.
 \end{remark}

Theorem~\ref{c:minC} follows from Theorem~\ref{min1}, Proposition~\ref{min2}, Proposition~\ref{min3}, Proposition~\ref{min4}.

\begin{remark}
  The Segre code of \cite{CL-long1}  mentioned in this paragraph
  had	already been previously studied by F. Rodier in~\cite{Rodier2003}. We were not aware of this work
  and we thank the anonymous reviewer of  the present paper for having brought it to our attention.
  In~\cite{Rodier2003} the author computed the parameters of the Segre
  code, as well as its weight spectrum and the number of codewords of minimum weight.
  The characterization of the weight of the codewords in~\cite{Rodier2003} and \cite{CL-long1} is
	also  the same.
\end{remark}

\section{Open problems}
There are several interesting open problems about the codes presented in
this paper; here we mention the main ones.
\begin{enumerate}
\item Determine some of the possible Hamming weights for Grassmann codes when $k>3$; in particular
  the second smallest weight of Grassmann codes has been recently determined
  by~\cite{Datta2025}, but little information is available about their weight
  spectrum. With respect to the minimality of codes, it is interesting to
  determine in general the \emph{maximum} weight of Grassmann codes.
\item Determine the minimum distance and maximum weight for
  polar Grassmann codes when $k>3$ in the symplectic, orthogonal
  and hermitian cases.
\item Consider in detail codes arising from the embeddings of
  flag geometries of type different from $(1,n-1)$. In particular
  the cases
  $(1,t)$ with $t<n-1$ and the case $(2,n-2)$ are of interest
  as starting points.
\item Study the codes arising from the  Pl\"ucker embedding of
  $J$-Grassmann geometries of polar spaces.
\item Provide efficient algorithms for the codes arising from geometries; in particular
  provide point-enumerators to facilitate the encoding and local decoding algorithms.
   For Grassmann codes, see~\cite{Beelen2021}. For polar Grassmann
  codes in the
  case $k=2$, see e.g. the approach of~\cite{Impl0, Impl1}.
\item Provide a geometric characterization of quasi-minimal codes.
\end{enumerate}

\section*{Thanks}
The authors thank the anonymous reviewer of the first version of this paper
for the careful reading of the paper and the precious suggestions offered.


\begin{thebibliography}{10}

\bibitem{ABN-2022}
Gianira~N. Alfarano, Martino Borello, and Alessandro Neri, \emph{A geometric
  characterization of minimal codes and their asymptotic performance}, Advances
  in Mathematics of Communications \textbf{16} (2022), no.~1, 115--133.
  \MR{4393725}

\bibitem{Alfarano2022}
Gianira~N. Alfarano, Martino Borello, Alessandro Neri, and Alberto Ravagnani,
  \emph{Three combinatorial perspectives on minimal codes}, SIAM Journal on
  Discrete Mathematics \textbf{36} (2022), no.~1, 461--489. \MR{4381528}

\bibitem{Alon2024}
Noga Alon, Anurag Bishnoi, Shagnik Das, and Alessandro Neri, \emph{Strong
  blocking sets and minimal codes from expander graphs}, Transactions of the
  American Mathematical Society \textbf{377} (2024), no.~8, 5389--5410.
  \MR{4771226}

\bibitem{Ashikhmin1998}
Alexei Ashikhmin and Alexander Barg, \emph{Minimal vectors in linear codes},
  Institute of Electrical and Electronics Engineers. Transactions on
  Information Theory \textbf{44} (1998), no.~5, 2010--2017. \MR{1664103}

\bibitem{BGH}
Peter Beelen, Sudhir~R. Ghorpade, and Sartaj~Ul Hasan, \emph{Linear codes
  associated to determinantal varieties}, Discrete Mathematics \textbf{338}
  (2015), no.~8, 1493--1500. \MR{3336120}

\bibitem{Beelen2021}
Peter Beelen and Prasant Singh, \emph{Point-line incidence on {G}rassmannians
  and majority logic decoding of {G}rassmann codes}, Finite Fields and their
  Applications \textbf{73} (2021), Paper No. 101843, 24. \MR{4241603}

\bibitem{Berlekamp1978}
Elwyn Berlekamp, Robert~J. McEliece, and Henk C.~A. van Tilborg, \emph{On the
  inherent intractability of certain coding problems (corresp.)}, IEEE
  Transactions on Information Theory \textbf{24} (1978), no.~3, 384--386.

\bibitem{Bishnoi2024}
Anurag Bishnoi, Jozefien D'haeseleer, Dion Gijswijt, and Aditya Potukuchi,
  \emph{Blocking sets, minimal codes and trifferent codes}, Journal of the
  London Mathematical Society. Second Series \textbf{109} (2024), no.~6, Paper
  No. e12938, 26. \MR{4751868}

\bibitem{Bonini2021}
Matteo Bonini and Martino Borello, \emph{Minimal linear codes arising from
  blocking sets}, Journal of Algebraic Combinatorics. An International Journal
  \textbf{53} (2021), no.~2, 327--341. \MR{4238182}

\bibitem{Bourbaki1998}
Nicolas Bourbaki, \emph{Algebra {I}. {C}hapters 1--3}, Elements of Mathematics
  (Berlin), Springer-Verlag, Berlin, 1998, Translated from the French, Reprint
  of the 1989 English translation. \MR{1727844}

\bibitem{Burton}
Robert~Corry Burton, \emph{An {A}pplication {of} {Convex} {Sets} {to} {the}
  {Construction} {of} {Error} {Correcting} {Codes} {and} {Factorial}
  {Designs}}, ProQuest LLC, Ann Arbor, MI, 1964, Thesis (Ph.D.)--The University
  of North Carolina at Chapel Hill. \MR{2614563}

\bibitem{Cardinali2009}
Ilaria Cardinali, Bart De~Bruyn, and Antonio Pasini, \emph{On the simple
  connectedness of hyperplane complements in dual polar spaces}, Discrete
  Mathematics \textbf{309} (2009), no.~2, 294--303. \MR{2473080}

\bibitem{Cardinali2013}
Ilaria Cardinali and Luca Giuzzi, \emph{Codes and caps from orthogonal
  {G}rassmannians}, Finite Fields and their Applications \textbf{24} (2013),
  148--169. \MR{3093864}

\bibitem{Cardinali2016b}
\bysame, \emph{Minimum distance of symplectic {G}rassmann codes}, Linear
  Algebra and its Applications \textbf{488} (2016), 124--134. \MR{3419777}

\bibitem{Impl0}
\bysame, \emph{Enumerative coding for line polar grassmannians with
  applications to codes}, Finite Fields and Their Applications \textbf{46}
  (2017), 107--138.

\bibitem{Cardinali2018}
\bysame, \emph{Line {H}ermitian {G}rassmann codes and their parameters}, Finite
  Fields and their Applications \textbf{51} (2018), 407--432. \MR{3781414}

\bibitem{Cardinali2018c}
\bysame, \emph{Minimum distance of orthogonal line-{G}rassmann codes in even
  characteristic}, Journal of Pure and Applied Algebra \textbf{222} (2018),
  no.~10, 2975--2988. \MR{3795630}

\bibitem{Impl1}
\bysame, \emph{Implementing line-hermitian grassmann codes}, Linear Algebra and
  its Applications \textbf{580} (2019), 96--120.

\bibitem{CL-long1}
\bysame, \emph{Linear codes arising from the point-hyperplane geometry-part I:
  The segre embedding}, Finite Fields and Their Applications \textbf{111}
  (2026), 102766.

\bibitem{CL-long2}
\bysame, \emph{Linear codes arising from the point-hyperplane geometry — part
  II: the twisted embedding}, Finite Fields and Their Applications \textbf{113}
  (2026), 102830.

\bibitem{Cardinali2016a}
Ilaria Cardinali, Luca Giuzzi, Krishna~V. Kaipa, and Antonio Pasini, \emph{Line
  polar {G}rassmann codes of orthogonal type}, Journal of Pure and Applied
  Algebra \textbf{220} (2016), no.~5, 1924--1934. \MR{3437272}

\bibitem{Cardinali2018a}
Ilaria Cardinali, Luca Giuzzi, and Antonio Pasini, \emph{On transparent
  embeddings of point-line geometries}, Journal of Combinatorial Theory. Series
  A \textbf{155} (2018), 190--224. \MR{3741426}

\bibitem{mexicans}
Jes\'us Carrillo-Pacheco and Felipe Zald\'ivar, \emph{On
  {L}agrangian-{G}rassmannian codes}, Designs, Codes and Cryptography. An
  International Journal \textbf{60} (2011), no.~3, 291--298. \MR{2800956}

\bibitem{Chen2025}
Hao Chen, Yaqi Chen, Conghui Xie, and Huimin Lao, \emph{Minimal linear codes
  violating the ashikhmin-barg condition from arbitrary projective linear
  codes},  (2025).

\bibitem{Chevalley1954}
Claude~C. Chevalley, \emph{The algebraic theory of spinors}, Columbia
  University Press, December 1954.

\bibitem{Cohen2013}
G\'erard~D. Cohen, Sihem Mesnager, and Alain Patey, \emph{On minimal and
  quasi-minimal linear codes}, Cryptography and coding, Lecture Notes in
  Comput. Sci., vol. 8308, Springer, Heidelberg, 2013, pp.~85--98. \MR{3163591}

\bibitem{Datta2025}
Mrinmoy Datta and Tiasa Dutta, \emph{The second minimum weight of grassmann
  codes},  (2025).

\bibitem{DGMP2011}
Alexander~A. Davydov, Massimo Giulietti, Stefano Marcugini, and Fernanda
  Pambianco, \emph{Linear nonbinary covering codes and saturating sets in
  projective spaces}, Advances in Mathematics of Communications \textbf{5}
  (2011), no.~1, 119--147. \MR{2770105}

\bibitem{Bart}
Bart De~Bruyn, \emph{Hyperplanes of embeddable {G}rassmannians arise from
  projective embeddings: a short proof}, Linear Algebra and its Applications
  \textbf{430} (2009), no.~1, 418--422. \MR{2460527}

\bibitem{Ghorpade2013}
Sudhir~R. Ghorpade and Krishna~V. Kaipa, \emph{Automorphism groups of
  {G}rassmann codes}, Finite Fields and their Applications \textbf{23} (2013),
  80--102. \MR{3061086}

\bibitem{Ghorpade2001}
Sudhir~R. Ghorpade and Gilles Lachaud, \emph{Hyperplane sections of
  {G}rassmannians and the number of {MDS} linear codes}, Finite Fields and
  their Applications \textbf{7} (2001), no.~4, 468--506. \MR{1866340}

\bibitem{Ghorpade2009}
Sudhir~R. Ghorpade, Arunkumar~R. Patil, and Harish~K. Pillai,
  \emph{Decomposable subspaces, linear sections of {G}rassmann varieties, and
  higher weights of {G}rassmann codes}, Finite Fields and their Applications
  \textbf{15} (2009), no.~1, 54--68. \MR{2468992}

\bibitem{Ghorpade2005}
Sudhir~R. Ghorpade and Michael~A. Tsfasman, \emph{Schubert varieties, linear
  codes and enumerative combinatorics}, Finite Fields and their Applications
  \textbf{11} (2005), no.~4, 684--699. \MR{2181414}

\bibitem{Gorla2023}
Elisa Gorla and Alberto Ravagnani, \emph{Generalized weights of codes over
  rings and invariants of monomial ideals}, Combinatorial Theory \textbf{3}
  (2023), no.~2.

\bibitem{Pinero2025}
Sarah Gregory, Fernando Pi\~nero Gonz\'alez, Doel Rivera-Laboy, and Lani
  Southern, \emph{Computing the minimum distance of the {$C(\Bbb O_{3,6})$}
  polar orthogonal {G}rassmann code with elementary methods}, Finite Fields and
  their Applications \textbf{107} (2025), Paper No. 102656, 28. \MR{4904929}

\bibitem{GGG}
James W.~P. Hirschfeld and Joseph~A. Thas, \emph{General {G}alois geometries},
  Springer Monographs in Mathematics, Springer, London, 2016. \MR{3445888}

\bibitem{Kaipa2013}
Krishna~V. Kaipa and Harish~K. Pillai, \emph{Weight spectrum of codes
  associated with the {G}rassmannian {$G(3,7)$}}, Institute of Electrical and
  Electronics Engineers. Transactions on Information Theory \textbf{59} (2013),
  no.~2, 986--993. \MR{3015710}

\bibitem{Kasikova}
Anna Kasikova, \emph{Simple connectedness of hyperplane complements in some
  geometries related to buildings}, Journal of Combinatorial Theory. Series A
  \textbf{118} (2011), no.~2, 641--671. \MR{2739510}

\bibitem{MacWilliams1977}
Florence~J. MacWilliams and Neil J.~A. Sloane, \emph{The theory of
  error-correcting codes.}, North-Holland Mathematical Library, vol. Vol. 16,
  North-Holland Publishing Co., Amsterdam-New York-Oxford, 1977. \MR{465510}

\bibitem{MacWilliams1962}
Florence~Jessie MacWilliams, \emph{Combinatorial properties of elementary
  abelian groups}, Harvard University, Cambridge, MA, 1962, Thesis
  (Ph.D.)--Radcliffe College. \MR{2939359}

\bibitem{Massey}
James~L. Massey, \emph{Minimal codewords and secret sharing}, Proceedings of
  the 6th Joint Swedish-Russian International Workshop on Information Theory,
  1993, pp.~276--279.

\bibitem{McInroy2010}
Justin McInroy and Sergey Shpectorov, \emph{On the simple connectedness of
  hyperplane complements in dual polar spaces. {II}}, Discrete Mathematics
  \textbf{310} (2010), no.~8, 1381--1388. \MR{2592492}

\bibitem{Nogin1996}
Dmitry~Yu. Nogin, \emph{Codes associated to {G}rassmannians}, Arithmetic,
  geometry and coding theory ({L}uminy, 1993), de Gruyter, Berlin, 1996,
  pp.~145--154. \MR{1394931}

\bibitem{Nogin1997}
\bysame, \emph{The spectrum of codes associated with the {G}rassmannian variety
  {$G(3,6)$}}, Rossi\u iskaya Akademiya Nauk. Problemy Peredachi Informatsii
  \textbf{33} (1997), no.~2, 26--36. \MR{1663924}

\bibitem{Pasalic2021}
Enes Pasalic, René Rodríguez, Fengrong Zhang, and Yongzhuang Wei,
  \emph{Several classes of minimal binary linear codes violating the
  ashikhmin-barg bound}, Cryptography and Communications \textbf{13} (2021),
  no.~5, 637--659.

\bibitem{Pasini24}
Antonio Pasini, \emph{Geometric hyperplanes of the {L}ie geometry
  ${A}_{n,\{1,n\}}(\mathbb {F})$}, Ricerche di Matematica \textbf{74} (2024),
  no.~3, 1321--1340. \MR{4930066}

\bibitem{Rodier2003}
François Rodier, \emph{Codes from flag varieties over a finite field}, Journal
  of Pure and Applied Algebra \textbf{178} (2003), no.~2, 203--214.

\bibitem{R1}
Charles~T. Ryan, \emph{An application of {G}rassmannian varieties to coding
  theory}, vol.~57, 1987, Sixteenth Manitoba conference on numerical
  mathematics and computing (Winnipeg, Man., 1986), pp.~257--271. \MR{889714}

\bibitem{R2}
\bysame, \emph{Projective codes based on {G}rassmann varieties}, vol.~57, 1987,
  Sixteenth Manitoba conference on numerical mathematics and computing
  (Winnipeg, Man., 1986), pp.~273--279. \MR{889715}

\bibitem{R4}
\bysame, \emph{The weight distribution of a code associated to intersection
  properties of the {G}rassmann variety {$G(3,6)$}}, vol.~61, 1988, Seventeenth
  Manitoba Conference on Numerical Mathematics and Computing (Winnipeg, MB,
  1987), pp.~183--198. \MR{961652}

\bibitem{R3}
Charles~T. Ryan and Kevin~M. Ryan, \emph{The minimum weight of the {G}rassmann
  codes {$C(k,n)$}}, Discrete Applied Mathematics. The Journal of Combinatorial
  Algorithms, Informatics and Computational Sciences \textbf{28} (1990), no.~2,
  149--156. \MR{1067604}

\bibitem{Scotti2024}
Martin Scotti, \emph{On the lower bound for the length of minimal codes},
  Discrete Mathematics \textbf{347} (2024), no.~1, Paper No. 113676, 10.
  \MR{4644284}

\bibitem{Shannon}
Claude~E. Shannon, \emph{A mathematical theory of communication}, The Bell
  System Technical Journal \textbf{27} (1948), 379--423, 623--656. \MR{26286}

\bibitem{Shult1995}
Ernest~E. Shult, \emph{Embeddings and hyperplanes of {L}ie incidence
  geometries}, Groups of {L}ie type and their geometries ({C}omo, 1993), London
  Math. Soc. Lecture Note Ser., vol. 207, Cambridge Univ. Press, Cambridge,
  1995, pp.~215--232. \MR{1320524}

\bibitem{Shult2011}
\bysame, \emph{Points and lines}, Universitext, Springer, Heidelberg, 2011,
  Characterizing the classical geometries. \MR{2761484}

\bibitem{ShultThas}
Ernest~E. Shult and Joseph~A. Thas, \emph{Hyperplanes of dual polar spaces and
  the spin module}, Archiv der Mathematik \textbf{59} (1992), no.~6, 610--623.
  \MR{1189881}

\bibitem{Tang2021}
Chunming Tang, Yan Qiu, Qunying Liao, and Zhengchun Zhou, \emph{Full
  characterization of minimal linear codes as cutting blocking sets}, Institute
  of Electrical and Electronics Engineers. Transactions on Information Theory
  \textbf{67} (2021), no.~6, 3690--3700. \MR{4289345}

\bibitem{TVZ}
Michael Tsfasman, Serge Vlăduţ, and Dmitry Nogin, \emph{Algebraic geometric
  codes: basic notions}, vol. 139, American Mathematical Society, September
  2007.

\bibitem{VanMaldeghem2024}
Hendrik Van~Maldeghem, \emph{Hyperplanes of {S}egre geometries}, Ars Combin.
  \textbf{160} (2024), 59--71. \MR{4808663}

\end{thebibliography}


\providecommand{\bysame}{\leavevmode\hbox to3em{\hrulefill}\thinspace}
\providecommand{\MR}{\relax\ifhmode\unskip\space\fi MR }
\providecommand{\MRhref}[2]{%
  \href{http://www.ams.org/mathscinet-getitem?mr=#1}{#2}
}
\providecommand{\href}[2]{#2}
\renewcommand{\MR}[1]{\relax}
\renewcommand{\MRhref}[2]{\relax}

\end{document}